\documentclass[review]{elsarticle}

\usepackage{lineno,hyperref}
\modulolinenumbers[5]

\journal{arXiv}

\usepackage[usenames]{color}
\usepackage{amssymb}
\usepackage{graphicx}
\usepackage{amscd}

\definecolor{webgreen}{rgb}{0,.5,0}
\definecolor{webbrown}{rgb}{.6,0,0}

\usepackage{color}
\usepackage{fullpage}
\usepackage{float}

\usepackage{psfig}
\usepackage{graphics,amsmath,amssymb}
\usepackage{amsthm}
\usepackage{amsfonts}
\usepackage{latexsym}
\usepackage{epsf}

\DeclareMathOperator{\perm}{perm}

\theoremstyle{plain}
\newtheorem{theorem}{Theorem}
\newtheorem{corollary}[theorem]{Corollary}
\newtheorem{lemma}[theorem]{Lemma}

\theoremstyle{definition}
\newtheorem{definition}[theorem]{Definition}

\theoremstyle{remark}
\newtheorem{remark}[theorem]{Remark}

\begin{document}

\begin{frontmatter}

\title{Rational Convolution Roots of Isobaric Polynomials}

\author[Conciaddress]{Aura Conci}
\address[Conciaddress]{Department of Computer Science, Universidade Federal Fluminense, 156, Sao Domingos
Niteroi - Rio de Janeiro, Brasil}
\ead[url]{http://www.ic.uff.br/$\sim$aconci/}
\ead{aconci@ic.uff.br}

\author[Liaddress]{Huilan Li\corref{mycorrespondingauthor}}
\address[Liaddress]{Department of Mathematics, Drexel University, 3141 Chestnut Street, Philadelphia PA, 19104 USA}
\cortext[mycorrespondingauthor]{Corresponding author}
\ead{huilan.li@gmail.com}


\ead[url]{http://www.math.drexel.edu/$\sim$huilan}

\author[MacHenryaddress]{Trueman MacHenry}
\address[MacHenryaddress]{Department of Mathematics and Statistics, York University, 4700 Keele Street, Toronto ON, M3J 1P3 Canada}

\ead{machenry@mathstat.yorku.ca}
\ead[url]{http://www.math.yorku.ca/$\sim$machenry/‎}




\begin{abstract}

In this  paper, we exhibit two matrix representations of the rational roots of generalized Fibonacci polynomials (GFPs) under convolution product,  in terms of determinants and  permanents, respectively.  The underlying root formulas for GFPs and for weighted isobaric polynomials (WIPs), which  appeared in an earlier paper by  MacHenry and Tudose, make use of two types of operators. These operators are derived from the generating functions for Stirling numbers of the first kind and second kind.  Hence we call them Stirling operators. To construct matrix representations of the roots of GFPs, we use the Stirling operators of the first kind.  We give explicit examples to show how the Stirling operators of the second kind appear in the low degree cases for the WIP-roots.
  As a consequence of the matrix construction, we have matrix representations of multiplicative arithmetic functions under the Dirichlet product  into its divisible closure.  
\end{abstract}


\end{frontmatter}


\section{Introduction}

In 1975,  Carroll and  Gioia \cite{CG} gave a direct construction for an adjoining $q$-th roots, $q\in \mathbb{Q}$, to the group of multiplicative arithmetic functions (MF) under  Dirichlet product.  In 2000, MacHenry \cite{M4} gave a somewhat more general proof of the same result. In 2005, MacHenry and Tudose  \cite{MT1} constructed the injective hull of the generalized Fibonacci polynomials (GFPs) and extended this construction to the injective hull of  the \emph{WIP-module},  that is,  the $\mathbb{Z}$-module of all sequences of weighted isobaric polynomials  (WIPs) with convolution product.  The isobaric polynomials are the symmetric polynomials over the elementary  symmetric polynomial (ESP) basis;  the isobaric ring is isomorphic to the ring of symmetric polynomials. In 2012, MacHenry and Wong \cite{MW2} showed that GFPs together with the convolution product give  a faithful representation of the group of MF under the Dirichlet product,  which in turn induces the embedding of the group  MF into its injective hull,  that is,  adjoins a $q$-th root to each multiplicative arithmetic function for all non-zero rational numbers $q$ in $ \mathbb{Q}.$ 

In 2013,  Li and MacHenry \cite{LM2} gave two matrix representations of WIPs in terms of Hessenberg matrices;  they showed that the determinant of one matrix is the permanent of the other,  and the determinant and permanent is an element in WIP-module.

In this paper,  we use Hessenberg matrices to give matrix representations of the $q$-th, $q\in\mathbb{Q}$, convolution roots of  GFPs,  both in terms of determinants and in terms of permanents.   The main result of this paper is Theorem \ref{my theoremRoots} in Section \ref{rootsGFP}. We introduce  $B_j = q(q+1)\cdots(q+j)$   and   $B_{-j} = q(q-1)\cdots(q-j)$, the \emph{Stirling operators of 1st kind and 2nd kind}, respectively. Then we get Corollary \ref{my corollaryStirlingOperators}:\\
$F^q_{k,n}$ is the determinant of the following matrix:
$$ \left(\begin{array}{cccccc}
B_0t_1 & -1 & 0 &0&  \cdots & 0 \\
B_0t_2 & \frac{1}{2} \frac{B_1}{B_0} t_1 & -1 &0&\cdots & 0 \\
B_0t_3 &   \frac{1}{3}(2 \frac{B_1}{B_0}-1) t_2 &  \frac{1}{3} \frac{B_2}{B_1} t_1& -1& \cdots & 0  \\
\vdots & \vdots & \vdots &\vdots& \ddots & \vdots\\
B_0t_n&   \hbox{\scriptsize $\frac{1}{n}\Big( (n-1)\frac{B_1}{B_0}-(n-2)\Big)t_{n-1}$}&  \hbox{\scriptsize $ \frac{1}{n}\Big( (n-2)\frac{B_2}{B_1}-2(n-3)\Big) t_{n-2}$}&  \hbox{\scriptsize $ \frac{1}{n}\Big( (n-3)\frac{B_3}{B_2}-3(n-4)\Big) t_{n-2}$}& \cdots &  \frac{1}{n}\Big( \frac{B_{n-1}}{B_{n-2}}\Big) t_1  \end{array}\right),$$
{\it i.e.}, the recursion coefficients are $$s_j = \frac{1}{n} \Big(j\frac{B_{n-j}}{B_{n-j-1}}-(n-j)(j-1)\Big)t_{j}, \mbox{ for } j=1,\ldots,n-1,  \mbox{ and } s_n=B_0t_n.$$ 
We call these representations \textit{Hessenberg-Stirling representations.} 

In order to  produce the convolution roots of GFPs and WIPs, we use the Stirling operators of the first kind and second kind.
 So far, we know of no such applications using the Stirling generating functions.  We would like to point out the unexpected usefulness of the Stirling operators. They give a complete answer to the construction of the rational roots of the  group of multiplicative arithmetic functions under the Dirichlet product  \cite{CG}, which is a concern arising in arithmetic number theory.  We describe how GFPs are used to produce an isomorphism from the group generated by GFPs under the convolution product to the group of multiplicative arithmetic functions under Dirichlet product \cite{LM1}.     The usefulness of Stirling operators in this case suggests that adjoining roots (and powers) to other algebraic structures may also be achieved by using Stirling operators.  Also the Stirling operators may have wider applications,  say, to other groups.  
 
 This paper is organized as follows. In Section \ref{isopoly}, we review basic facts about isobaric polynomials. In Section \ref{Hessenbergrep}, we review the Hessenberg representations of WIPs. In Section \ref{rootsGFP}, we construct matrix representations of roots of GFPs using Stirling operators. In Section \ref{WIP-MF}, we review the isomorphism from the group generated by GFPs to the group of multiplicative arithmetic functions. We thus have matrix representations of the elements in the divisible closure of MF under Dirichlet product.

 \section{Isobaric Polynomials}\label{isopoly}

An \emph{isobaric polynomial} in $k$ variables $\{t_1,\ldots,t_k\}$ of degree $n$ is of this form
$$P_{k,n} = \sum_{ \alpha   \vdash  n} C_\alpha t_1^{\alpha_1}  t_2^{\alpha_2}\cdots t_k^{\alpha_k}, $$
where $C_\alpha \in \mathbb{Z}$ and $\alpha= (\alpha_1,\alpha_2,\ldots,\alpha_k) \vdash n$ means that $(1^{\alpha_1}, 2^{\alpha_2}, \ldots, k^{\alpha_k}) $ is a partition of  $n$ with  $\displaystyle\sum_{j=1}^k j\alpha_j = n$. One can think of an isobaric polynomial as a symmetric polynomial written on the elementary symmetric polynomial (ESP) basis.

A sequence of \textit{weighted isobaric polynomials} of weight  $\omega = (\omega_1, \omega_2,\ldots, \omega_j, \ldots )$ with $\omega_j \in\mathbb{Z}$ is defined by  
$$P_{\omega,k,n} = \sum_{\alpha \vdash n} \left(\begin{array}{c}|\alpha| \\\alpha_1,\ldots,\alpha_k\end{array}\right) \frac{\sum \alpha_i\omega_i}{|\alpha|} t_1^{\alpha_1}  t_2^{\alpha_2}\cdots t_k^{\alpha_k },$$
where $|\alpha|=\alpha_1+\alpha_2+\cdots+\alpha_k.$ The union of elements of all sequences of weighted isobaric polynomials is the set of all isobaric polynomials.
The indexing set $\{n\} $ for these polynomials is the set of integers,  positive, negative and  zero, \textit{ i.e.}, $n \in\mathbb{Z}$.  In particular,  $P_{\omega,k,0} = \omega_k, k \geqslant 1,$  and $ P_{\omega,k,0} =1, k=0$  \cite{LM1}.

Note that the monomials are indexed by the partitions $(1^{\alpha_1}, 2^{\alpha_2}, \ldots,k^{\alpha_k})$ with parts no larger than $k$. Moreover,  the elements in a sequence of weighted isobaric polynomials occur in linear recursions
$$P_{\omega,k,n} =  t_1P_{\omega,k,n-1} + t_2 P_{\omega,k,n-2} + \cdots + t_j P_{\omega,k,n-j} + \cdots + t_kP_{\omega,k,n-k},$$  with respect to the recursion parameters  $ [t_1,\ldots, t_k] $.

Two importance  sequences are the generalized Fibonacci polynomials (GFPs)
$$F_{k,n} = \sum_{\alpha \vdash n} \left(\begin{array}{c}|\alpha| \\\alpha_1,\ldots,\alpha_k\end{array}\right) t_1^{\alpha_1}  t_2^{\alpha_2}\cdots t_k^{\alpha_k },$$
where the weight vector is $\omega = (1,1,\ldots,1\ldots)$ with  $F_{k,0} = 1$, and the generalized Lucas polynomials (GLPs)
$$G_{k,n} = \sum_{\alpha \vdash n} \left(\begin{array}{c}|\alpha| \\\alpha_1,\ldots,\alpha_k\end{array}\right) \frac{n}{|\alpha|} t_1^{\alpha_1}  t_2^{\alpha_2}\cdots t_k^{\alpha_k },$$
where the weight vector is  $\omega = (1,2,\ldots,j,\ldots) $ and  $G_{k,0} = k.$

\begin{remark}  The GFPs are the complete symmetric polynomials written on the ESP basis,  and the GLPs are the power sum symmetric polynomials written on the ESP basis; each of these sequences of polynomials is a basis for the ring of symmetric polynomials.
\end{remark}

\begin{remark}  The WIPs in general have special significance in the ring of symmetric polynomials.  In order to see how this comes about,  it is convenient to consider the notation $ [t_1,\ldots, t_k] $ used above  to indicate recursion parameters.  More generally,  we also use $ [t_1,\ldots, t_k] $ to indicate the monic polynomial  $\mathcal{C}(X) = X^k -t_1X^{k-1}-\cdots-t_k,$
that is, $$[t_1,\ldots, t_k] = X^k -t_1X^{k-1}-\cdots-t_k.$$
When we consider  $t_j$'s  as variables,  we often refer to $\mathcal{C}(X) $ as the \textit{generic core},  and when we evaluate   $t_j$'s over the ring of integers,  the term \textit{numerical core}  will   be used.  
\end{remark}

\begin{remark}  It is trivial to verify that   when $k = 2$,   the GFPs are generalizations of the classical ``generalized Fibonacci polynomials'' and the GLPs are generalizations of the classical ``generalized Lucas polynomials' \cite{ACMS};  when $t_1=  t_2=1$,  the GFPs and GLPs become the classical Fibonacci and Lucas sequences. It is also surprising that these older terms persist in the current literature in competition with the true generalizations.  
\end{remark}

Next,  we consider the \emph{companion matrix} of  $[t_1,\ldots, t_k] = X^k -t_1X^{k-1}-\cdots-t_k$,  namely,  the $k\times k$-matrix
$$ A_k=\left(\begin{array}{ccccc} 0& 1 & 0 & 0 & 0 \\0 & 0 & 1 & 0 & 0 \\\vdots & \vdots & \ddots & \vdots & \vdots \\0 & 0 & \cdots & 0 & 1 \\t_k & t_{k-1} & \cdots & t_2 & t_1\end{array}\right).$$
We use  $A_k$ to construct the following infinite  matrix by appending the orbit of the row vectors generated by letting    $A_k$  act on the right of the last row vector in  $A_k$,  and repeating the process on the successive last row vectors.  Noting that $A_k$ is non-singular exactly when  $t_k \neq 0$ and adding this as an assumption,  we can perform the analogous operation on the first row vector of   $A_k$,  extending the rows northward,  yielding a doubly infinite matrix with  $k$ columns.
 We call this the \textit{infinite companion matrix},  and denote it by  $A_k^\infty$,  or simply as $A^\infty$ when the  $k$  is clear.    Since it is completely determined by the polynomial  $\mathcal{C}(X) =[t_1,\ldots, t_k] $  we call $C(X)$  the \textit{core polynomial}.
  $${\Large A_k^\infty} = \left( \begin{matrix}

     \vdots &  \ddots & \vdots & \vdots \\
      (-1)^{k-1}S_{(-2,1^{(k-1)})} & \cdots  & -S_{(-2,1)} & S_{(-2)}  \\
      (-1)^{k-1}S_{(-1,1^{(k-1)})} & \cdots  & -S_{(-1,1)} & S_{(-1)}  \\
      (-1)^{k-1}S_{(0,1^{(k-1)}))}  &\cdots   &  -S_{(0,1)} & S_{(0)}  \\
      (-1)^{k-1}S_{(1,1^{(k-1)})}  & \cdots  &  -S_{(1,1)} & S_{(1)}  \\
      (-1)^{k-1}S_{(2,1^{(k-1)})}  & \cdots   &  -S_{(2,1)} & S_{(2)}  \\
      (-1)^{k-1}S_{(3,1^{(k-1)})}  & \cdots   &  -S_{(3,1)} & S_{(3)}  \\
      (-1)^{k-1}S_{(4,1^{(k-1)})} & \cdots   &  -S_{(4,1)} & S_{(4)}  \\
                    \vdots                      &  \ddots    &    \vdots     &    \vdots

   \end{matrix} \right) =  {\Big((-1)^{k-j} S_{(n,1^{k-j})}\Big).} $$
As pointed out in a number of previous papers (\textit{e.g.}, \cite{LM1})  the matrix $A_k^\infty$ has the following remarkable properties:
\begin{itemize}
  \item  The $k\times k$ contiguous blocks of  $A^\infty$ are the successive powers in the free abelian group generated by the companion matrix   $A_k$.
  \item The rows of $A^\infty$ give a vector representation of the successive powers of  the zeros of the core polynomial. This is essentially a consequence of the Hamilton-Cayley theorem. 
  \item The right hand column of   $A^\infty$ is just the GFPs. 
  \item 
The traces of the $k\times k$ contiguous blocks give in succession the GLPs. 
\item The  $k$ columns of $A^{\infty}$ are linearly recursive with respect to the coefficients of the core polynomial as recursion parameters.        
\item The columns of $A^\infty$ are sequences of  weighted isobaric polynomials with weights $\pm
(0,\ldots,0,1,1,\ldots,1,\ldots).$
\item The elements of $A^\infty$ are Schur-hook polynomials  $S_{(n,1^r)}$  of arm-length $n-1$ and leg length $r$.  In particular,  $F_{k,n}= S_{(n)}$.
\item  The sequences of WIPs form a free $\mathbb{Z}$-module. The columns of $A^\infty$ form a basis of this module.
\end{itemize}

Moreover,  there is a second matrix that is induced by the core polynomial \cite{LM1}.
Consider the derivative of the core polynomial
$$\mathcal{\mathcal{C}}'(X)= k X^{k-1} -t_1 X^{k-2}-\cdots - t_{k-1},$$
out of which,  we manufacture the vector  $(-t_{k-1}, \ldots, -t_1, k)$.  Again letting the companion matrix  $A_k$ act on this vector on the right and appending the resulting orbit as additional row vectors,  we get a $ k\times k$-matrix, which we call the \textit{different matrix} denoted by $D$.   From $D$ we construct an infinite matrix,
$D^\infty$, as we do for $A^{\infty}$.  We call $D^\infty$ the \textit{infinite different matrix}.  It too has some useful and remarkable properties  \cite{LM1}:

\begin{itemize}
  \item  The determinant  $\det D  = \Delta $, the discriminant of the core polynomial.
  \item  The right hand column of  $D^\infty$ is the sequence  GLPs. 
  \item  There is a bijection $\mathcal{L}$ from  $A^\infty$ to  $D^\infty$ which takes  the element $a_{i,j}$ in $A^\infty$ to $d_{i,j}$ in $D^\infty$,  which has the properties of a logarithm on elements, and which implies that $\mathcal{L} (F_{k,n}) = G_{k,n} $.
 \item  The columns of $D^\infty$ are linear recursions with recursion parameters  $\{t_1, \ldots,t_k\}$.
\end{itemize}


Next,  we would like to point out in what way the sequences discussed here are important.  In a series of papers \cite{MT1,MT2,MW1,MW2,LM1,LM2} it has been shown that  subgroups of the ring of arithmetic functions, namely. the Dirichlet group of multiplicative arithmetic functions,  and the additive group of additive arithmetic  functions have faithful representations using the GFP sequence and the GLP sequence;  they also show up in the character theory of the symmetric groups and P\'olya's  Theory of Counting \cite{P,LM1}.  In the following section,  we will recall the matrix representations of the GFP, the GLP and in general, the WIP sequences \cite{LM2},  which give an explicit algorithm to compute these sequences and are useful for calculation.

But first it is convenient to introduce the convolution product of weighted isobaric polynomials.

\begin{definition}[\cite{LM1}]  Let  $P_{\omega, k,n}$ and $P_{\upsilon, k,n}$ be  weighted isobaric polynomials of isobaric degree $n$. Define  the \emph{convolution product} of  $P_{\omega,k,n}$ and $P_{\upsilon, k,n}$ by  $$P_{\omega,k,n} \ast P_{\upsilon,k,n} =  \sum_{j=0}^n  P_{\omega,k,j}  P_{\upsilon,k,n-j}.$$
\end{definition}

Note that the product is also a weighted isobaric polynomial of isobaric degree $n$.  In the case where we have two integer evaluations  of   $P_{\omega, k,n}$ and $P_{\upsilon, k,n}$, we denote them as, respectively,  $P'$ and $P''$,  and their numerical convolution product is    $$P'_{\omega,k,n} \ast P''_{\upsilon,k,n} =  \sum_{j=0}^n  P'_{\omega,k,j}  P''_{\upsilon,k, n-j}. $$  
It is with respect to this product and the ordinary addition of polynomials that the  \emph{logarithm operator} $\mathcal{L}$  is defined (see more details in\cite{LM1}).

\section{Permanent and Determinant Representations}\label{Hessenbergrep}

A formula was given for the elements  of the divisible closure of the WIP-module \cite{MT1}, \textit{i.e.},  each element in  WIP-module was given a $q$-th root for all   $q \in\mathbb{Q}$,  where these roots are unique up to sign.   Two interesting representations of the elements of  WIP-module were given in terms of determinants and  permanents of the following Hessenberg matrices \cite{LM2}:
 $$H_{+(\omega,k,n)} = \left(\begin{matrix}
t_1 & 1 & 0 & \cdots & 0 \\
t_2 & t_1 & 1 & \cdots & 0 \\
\vdots & \vdots & \vdots & \ddots & \vdots \\
t_{n-1} & t_{n-2} & t_{n-3} & \cdots & 1 \\
\omega_n t_n & \omega_{n-1} t_{n-1} & \omega_{n-2}t_{n-2}  & \cdots & \omega_1t_1
\end{matrix} \right),$$
and
$$H_{-(\omega,k,n)} = \left(\begin{matrix}
t_1 & -1 & 0 & \cdots & 0 \\
t_2 & t_1 & -1 & \cdots & 0 \\
\vdots & \vdots & \vdots & \ddots & \vdots  
t_{n-1} & t_{n-2} & t_{n-3} & \cdots & -1 \\
\omega_n t_n & \omega_{n-1} t_{n-1} & \omega_{n-2}t_{n-2}  & \cdots & \omega_1t_1
\end{matrix} \right).$$
The principal results are that   $$\perm H_{+(\omega, k, n)} = P_{\omega,k,n} =\det H_{-(\omega,k,n)}.$$

For example, we look at the following matrix when  $n=4$
$$ \left( \begin{matrix}
t_1 & 1 & 0 & 0 \\ t_2 & t_1 & 1 & 0 \\ t_3 & t_2 & t_1 & 1 \\ \omega_4 t_4 & \omega_3 t_3 & \omega_2 t_2 & \omega_1 t_1
\end{matrix}\right)$$
whose permanent is easily seen to be \; $\omega_1t_1^4 + (2\omega_1 +\omega_2)t_1^2t_2 + \omega_2t_2^2 + (\omega_1 + \omega_3)t_1t_3 +\omega_4t_4 = P_{\omega,4,4}.$

Moreover,  it is easy to see that there is a nesting of the Hessenberg matrices from lower right hand corner to upper left. We call these representations  \textit{Hessenberg representations}. It turns out that we can use these to go further and give two useful representations of the $q$-th convolution roots of  generalized Fibonacci polynomials in terms of Hessenberg matrices.


\section{Convolution Roots}\label{rootsGFP}

MacHenry and Tudose \cite[Theorems 5.1 and 5.7]{MT1} gave a general expression for the $q$-th, $q \in \mathbb{Q},$ convolution roots of the GFPs,  and a more general expression for the  $q$-th convolution roots  of  WIPs.  

The formula for $q$-th roots of  polynomials in  GFP
is given by  $$F^q_{k,n} =  \sum_{\alpha \vdash n}\frac{1}{|\alpha|!} B_{|\alpha|-1} \left(\begin{array}{c}|\alpha| \\\alpha_1,\ldots,\alpha_k\end{array}\right) t_1^{\alpha_1}  t_2^{\alpha_2}\cdots t_k^{\alpha_k }.$$
For $n=3$ and $n=4$,  we have the following determinantal representations
$$F^q_{k,3}=  \det \left(\begin{array}{ccc}qt_1 & -1 & 0 \\qt_2 & \frac{1}{2}(q+1)t_1 & -1 \\qt_3 &\frac{1}{3}(2q+1)t_2 & \frac{1}{3}(q+2)t_1\end{array}\right)$$ and
$$F^q_{k,4} = \det \left(\begin{array}{cccc}qt_1 & -1 & 0 & 0 \\qt_2 & \frac{1}{2}(q+1)t_1 & -1 & 0 \\qt_3 & \frac{1}{3}(2q+1)t_2 & \frac{1}{3}(q+2)t_1 & -1 \\qt_4 & \frac{1}{4}(3q+1)t_3 & \frac{1}{4}(2q+2)t_2 & \frac{1}{4}(q+3)t_1\end{array}\right),$$
 where  $B_j $ is   the polynomial generating function for the Stirling numbers of the 1st kind evaluated at $q$. ($B_{-j}$ the analogue,  determined by the polynomial generating function for Stirling numbers of the 2nd kind);  namely,  $B_j = q(q+1)\cdots(q+j)$   and   $B_{-j} = q(q-1)\cdots(q-j)$.  We call $B_j$ and $B_{-j}$ \emph{Stirling operators of 1st kind and 2nd kind}, respectively. 
 
 The main theorem of this paper is a generalization to arbitrary $n$ of the two matrices which appear above.

The first five such roots,  starting with  $F_{k,0}^{q}$ for an arbitrary  $q$ are

$\begin{array}{rcl}
  F_{k,0}^{q}& =& 1\\ 
  F_{k,1}^{q}& = &qt_1 \\
   F_{k,2}^{q}& = &\frac{1}{2}q(q+1)t_1^2 + qt_2\\ 
  F_{k,3}^{q}& = &\frac{1}{3!}q(q+1)(q+2)t_1^3 + q(q+1)t_1t_2+qt_3\\
   F_{k,4}^{q}& =&  \frac{1}{4!}q(q+1)(q+2)(q+3)t_1^4 +  \frac{1}{2}q(q+1)(q+2)t_1^2t_2 + \frac{1}{2}q(q+1)t_2^2 + q(q+1)t_1t_3 +qt_4\\
   F_{k,5}^{q}& = &\frac{1}{5!}q(q+1)(q+2)(q+3)(q+4)t_1^5 +  \frac{1}{6}q(q+1)(q+2)(q+3)t_1^3t_2 +  \frac{1}{2}q(q+1)(q+2)t_1t_2^2 \\
&&  +\frac{1}{2}q(q+1)(q+2) t_1^2t_3 +q(q+1) t_2t_3 + q(q+1) t_1t_4 + qt_5
\end{array}$
\\and in the Stirling operator notation,  these translate into:

$\begin{array}{rcl}
F^{q}_{k,0} &=& 1\\
F^{q}_{k,1} &=& B_0t_1\\
F^{q}_{k.2} &=& \frac{1}{2!} B_1 t^2_1 + B_0t_2\\
 F^{q}_{k, 3}& =& \frac{1}{3!} B_2 t^3_1 + \frac{1}{2!} 2 B_1 t_1 t_2+ B_0t_3\\
 F^{q}_{k,4}  &=& \frac{1}{4!} B_3 t^4_1 + \frac{1}{3!} 3 B_2 t^2_1 t_2+ \frac{1}{2!}B_1t^2_2+\frac{1}{2!}2B_1t_1t_3+B_0t_4\\
F^{q}_{k,5}  &=& \frac{1}{5!} B_4 t^5_1 + \frac{1}{4!} 4 B_3 t^3_1 t_2 + \frac{1}{3!}3 B_2t_1 t^2_2 + 
\frac{1}{3!} 3B_2 t^2_1t_3 +\frac{1}{2!}2 B_1 t_2t_3 + \frac{1}{2!}2 B_1t_1t_4 + B_0t_5
\end{array}$

A rule of thumb for writing the $q$-th convolution roots is as follows:  First write the polynomial $F_n$
as a function of $t_j, j = 1, \ldots, k$,  then, observing the exponent sum $|\alpha|$, monomial by monomial,  enter the fraction $ \frac{1}{|\alpha|!}$, and the Stirling operators $B_{|\alpha|- 1}$.  There will usually be some cancellations among the fractions for the most economical expression.

For example, $$F_{k,3} = t^3_1 + 2t_1t_2+t_3$$  and
$$ F^{q}_{k, 3} = \frac{1}{3!} B_2 t^3_1 + \frac{1}{2!} 2 B_1 t_1 t_2+ B_0t_3.$$

\begin{theorem}\label{my theoremRoots}   
 $$\begin{array}{rcl}
F^q_{k,n} &=&\det \left(\begin{array}{ccccc}qt_1 & -1 & 0 & \cdots & 0 \\qt_2 & \frac{1}{2}(q+1)t_1 & -1 & \cdots & 0 \\\vdots & \vdots & \vdots & \ddots & \vdots \\qt_{n-1} 
& \frac{1}{n-1}((n-2)q+1)t_{n-2}& \frac{1}{n-1}((n-3)q+2)t_{n-3}& \cdots & -1\\qt_n & \frac{1}{n}((n-1)q+1)t_{n-1} & \frac{1}{n}((n-2)q+2)t_{n-2} & \cdots & \frac{1}{n}(q+(n-1))t_1\end{array}\right)\\
&&\\
 &=&\perm \left(\begin{array}{ccccc}qt_1 & 1 & 0 & \cdots & 0 \\qt_2 & \frac{1}{2}(q+1)t_1 & 1 & \cdots & 0 \\\vdots & \vdots & \vdots & \ddots & \vdots \\qt_{n-1} 
& \frac{1}{n-1}((n-2)q+1)t_{n-2}& \frac{1}{n-1}((n-3)q+2)t_{n-3}& \cdots & 1\\qt_n & \frac{1}{n}((n-1)q+1)t_
{n-1} & \frac{1}{n}((n-2)q+2)t_{n-2} & \cdots & \frac{1}{n}(q+(n-1))t_1\end{array}\right).
\end{array}$$

\end{theorem}

\proof  Note that  the determinants and permanents are nested,  that is,  $F^q_{k,j} $ is the $ j \times j  $ principal minor in the upper left hand corner of  the matrices. This allows us to use induction in our proof.  We shall carry out the computations for the determinant case.  The proof for permanent case is similar. 

\begin{lemma}\label{mylemmatheorem1proof}  $F^{q}_{k,n}$ satisfies the recursive formula $$F^{q}_{k,n} = s_1F^{q}_{k,n-1} + s_2F^{q}_{n-2} + s_3F^{q}_{k,n-3} + \cdots + s_{n-1}F^{q}_{k,1} + s_nF^{q}_{k,0},$$ where the recursion parameters $s_j= \frac{1}{n}(jq+n-j)t_{j}$.

\end{lemma}

\begin{proof}  The nesting of the matrices,  and hence of the determinants and permanents, implies the recursion.
Let $
M_n=\det \left(\begin{array}{ccccc}qt_1 & -1 & 0 & \cdots & 0 \\qt_2 & \frac{1}{2}(q+1)t_1 & -1 & \cdots & 0 \\\vdots & \vdots & \vdots & \ddots & \vdots \\qt_{n-1} 
& \frac{1}{n-1}((n-2)q+1)t_{n-2}& \frac{1}{n-1}((n-3)q+2)t_{n-3}& \cdots & -1\\qt_n & \frac{1}{n}((n-1)q+1)t_{n-1} & \frac{1}{n}((n-2)q+2)t_{n-2} & \cdots & \frac{1}{n}(q+(n-1))t_1\end{array}\right) $, 
\\$M_0=1$ and $m_{i,j}$ be the $(i,j)$th entry in the matrix.
 To prove the recursive formula, it is equivalent to prove
$$M_n=m_{n,n}M_{n-1}+m_{n,n-1}M_{n-2}+\cdots+m_{n,2}M_1+m_{n,1}M_0.$$
We do the cofactor expansion along $n$th column from bottom to top and we get:
$$\begin{array}{ccl}
M_n&=&m_{n,n}M_{n-1}+\det \left(\begin{array}{ccccc}
m_{1,1}&-1&0&\cdots&0\\
m_{2,1}&m_{2,2}&-1&\cdots&0\\
\vdots&\vdots&\vdots&\ddots&\vdots\\
m_{n-2,1}&m_{n-2,2}&m_{n-2,3}&\cdots&-1\\
m_{n,1}&m_{n,2}&m_{n,3}&\cdots&m_{n,n-1}\\
\end{array}\right)\\
&&\hbox{

Then do the cofactor expansion along the last column from bottom to top.}\\

&=&m_{n,n}M_{n-1}+m_{n,n-1}M_{n-2}\\
&&\hskip 2cm+\det \left(\begin{array}{ccccc}
m_{1,1}&-1&0&\cdots&0\\
m_{2,1}&m_{2,2}&-1&\cdots&0\\
\vdots&\vdots&\vdots&\ddots&\vdots\\
m_{n-3,1}&m_{n-3,2}&m_{n-3,3}&\cdots&-1\\
m_{n,1}&m_{n,2}&m_{n,3}&\cdots&m_{n,n-2}\\
\end{array}\right)\\

&&\hbox{Continue to do cofactor along the  last column.}\\
&\vdots&\\
&=&m_{n,n}M_{n-1}+m_{n,n-1}M_{n-2}+\cdots+m_{n,2}M_1+m_{n,1}M_0.
\end{array}$$

Let   $s_j = m_{n,n-j+1.}$  We then have
$$M_n = s_1M_{n-1}+s_2M_{n-2}+\cdots+s_{n-1}M_1+s_nM_0.$$  Putting   $F^{q}_{k,n-j} = M_{n-j}$, we assume inductively that  $M_{n-j }  =  F^{q}_{k,n-j},  j = 0,1,\ldots,n-1$, we have $$M_n =  s_1F^{q}_{k,n-1} +s_2F^{q}_{k,n-2} +\cdots+s_{n-1}F^{q}_{k,1} +s_nF^{q}_{k,0} .$$ Now  we  only need to show that  $M_n = F^{q}_{k,n}$.


Recall that $\alpha=(\alpha_1,\alpha_2,\ldots,\alpha_k)\vdash n$ means $\alpha_1+2\alpha_2+\cdots+k\alpha_k=n$ and $|\alpha|=\alpha_1+\alpha_2+\cdots+\alpha_k$.

To prove $ F^{q}_{k,n} = \displaystyle\sum_{\alpha \vdash n}\frac{1}{|\alpha|!} B_{|\alpha|-1} \left(\begin{array}{c}|\alpha| \\\alpha_1,\ldots,\alpha_k\end{array}\right) t_1^{\alpha_1}  t_2^{\alpha_2}\cdots t_k^{\alpha_k }= \sum_{\alpha \vdash n}\frac{B_{|\alpha|-1}}{\alpha_1!\alpha_2!\cdots\alpha_k!}  t_1^{\alpha_1}  t_2^{\alpha_2}\cdots t_k^{\alpha_k }$ is the determinant or permanent of the matrices in Theorem \ref{my theoremRoots}  , we only need to show that $ F^{q}_{k,n} $ satisfies the recursive formula in Lemma \ref{mylemmatheorem1proof}, which is equivalent to showing that 
$$\frac{B_{|\alpha|-1}}{\alpha_1!\alpha_2!\cdots\alpha_k!} =\sum_{i=1}^k\frac{1}{n}(iq+n-i)\frac{B_{|\alpha|-2}}{\alpha_1!\alpha_2!\cdots(\alpha_{i}-1)!\cdots\alpha_k!}. $$
\begin{align*}
&\displaystyle\sum_{i=1}^k\frac{1}{n}(iq+n-i)\frac{B_{|\alpha|-2}}{\alpha_1!\alpha_2!\cdots(\alpha_{i}-1)!\cdots\alpha_k!}\\ =&\displaystyle\frac{(q+n-1)B_{|\alpha|-2}}{n(\alpha_1-1)!\alpha_2!\cdots\alpha_k!}+\frac{(2q+n-2)B_{|\alpha|-2}}{n\alpha_1!(\alpha_2-1)!\cdots\alpha_k!}+\cdots+\frac{(kq+n-k)B_{|\alpha|-2}}{n\alpha_1!\alpha_2!\cdots(\alpha_k-1)!}\\
=&\displaystyle\frac{B_{|\alpha|-2}}{n\alpha_1!\alpha_2!\cdots\alpha_k!}\Big[\alpha_1(q+n-1)+\alpha_2(2q+n-2)+\cdots+\alpha_k(kq+n-k)\Big]\\
=&\displaystyle\frac{B_{|\alpha|-2}}{n\alpha_1!\alpha_2!\cdots\alpha_k!}\Big[\alpha_1q+n\alpha_1-\alpha_1+2\alpha_2q+n\alpha_2-2\alpha_2+\cdots+k\alpha_kq+n\alpha_k-k\alpha_k\Big]\\
=&\displaystyle\frac{B_{|\alpha|-2}}{n\alpha_1!\alpha_2!\cdots\alpha_k!}\Big[(\alpha_1+2\alpha_2+\cdots+k\alpha_k)q+n(\alpha_1+\alpha_2+\cdots+\alpha_k)-(\alpha_1+2\alpha_2+\cdots+k\alpha_k)\Big]\\
=&\displaystyle\frac{B_{|\alpha|-2}}{n\alpha_1!\alpha_2!\cdots\alpha_k!}\Big(nq+n|\alpha|-n\Big)\\
=&\displaystyle\frac{B_{|\alpha|-2}(q+|\alpha|-1)}{\alpha_1!\alpha_2!\cdots\alpha_k!}\\
=&\displaystyle\frac{B_{|\alpha|-1}}{\alpha_1!\alpha_2!\cdots\alpha_k!}\\
\end{align*}
\end{proof}

It is of interest to see the matrix which represents the convolution roots in a form which explicitly displays the Stirling operators $B{_j}$,  which we now do in 

\begin{corollary}\label{my corollaryStirlingOperators} 

$F^q_{k,n}$ is the determinant of the following matrix:

\noindent$ \left(\begin{array}{cccccc}
B_0t_1 & -1 & 0 &0&  \cdots & 0 \\
B_0t_2 & \frac{1}{2} \frac{B_1}{B_0} t_1 & -1 &0&\cdots & 0 \\
B_0t_3 &   \frac{1}{3}(2 \frac{B_1}{B_0}-1) t_2 &  \frac{1}{3} \frac{B_2}{B_1} t_1& -1& \cdots & 0  \\
\vdots & \vdots & \vdots &\vdots& \ddots & \vdots\\
B_0t_n&   \hbox{\scriptsize $\frac{1}{n}\Big( (n-1)\frac{B_1}{B_0}-(n-2)\Big)t_{n-1}$}&  \hbox{\scriptsize $ \frac{1}{n}\Big( (n-2)\frac{B_2}{B_1}-2(n-3)\Big) t_{n-2}$}&  \hbox{\scriptsize $ \frac{1}{n}\Big( (n-3)\frac{B_3}{B_2}-3(n-4)\Big) t_{n-2}$}& \cdots &  \frac{1}{n}\Big( \frac{B_{n-1}}{B_{n-2}}\Big) t_1  \end{array}\right),$
i.e.,  the recursion coefficients are $$s_j = \frac{1}{n} \Big(j\frac{B_{n-j}}{B_{n-j-1}}-(n-j)(j-1)\Big)t_{j}, \mbox{ for } j=1,\ldots,n-1,  \mbox{ and } s_n=B_0t_n.$$ $\square$

\end{corollary}

We call these representations \textit{Hessenberg-Stirling representations.} The Stirling part is due to the role that the Stirling operators play in the construction of the roots of the GFPs.

The root formula for the WIPs  is a generalization of the root  formula for the GFP, and is a bit more complicated. 
\begin{theorem}[\cite{MT1}]\label{my theoremWIPs} 
$$P^q_{\omega,k,n}  =  \sum_{\alpha \vdash n}  L_{k,n,\omega}(\alpha) t^{\alpha_1}_1\cdots t^{\alpha_k}_k,   $$
where
$$ L_{\omega,k,n}(\alpha) =  \sum^{|\alpha|-1}_{j=0} \frac{1}{(\Pi_{i=1}^k\alpha_i)!}  \left(\begin{array}{c}|\alpha|-1 \\j\end{array}\right) B_{-j} D_{|\alpha|-j-1}( \omega^{\alpha_1}_1\cdots \omega^{\alpha_k}_k)$$

\noindent and  $D_j(\omega^{\alpha}) =  D_{j} (\omega^{\alpha_1}_1\cdots \omega^{\alpha_k}_k) $ is  the total derivative of the expression  $j$  times.

\end{theorem}

The \emph{total differential operator} $D_j$ is defined inductively by $D_j=D_1(D_{j-1})$ with $D_1(\omega^{\alpha_1}_1\cdots \omega^{\alpha_k}_k)=\displaystyle\sum_{i=1}^k\partial_i(\omega^{\alpha_1}_1\cdots \omega^{\alpha_k}_k)=\displaystyle\sum_{i=1}^k\alpha_i(\omega^{\alpha_1}_1\cdots \omega^{\alpha_i-1}_i\cdots \omega^{\alpha_k}_k)$.
 For example,  $D_2(\omega_1^3\omega_2^2) = 6 \omega_1\omega_2^2+12\omega_1^2\omega_2+2\omega_1^3.$

Here we give some low-dimensional examples for the $q$-th roots of weighted isobaric polynomials:

$\begin{array}{rcl}
P^q_{\omega, k,0} &=&1\\
P^q_{\omega,k,1} &= & q\omega_1 t_1\\
P_{\omega,k,2}^q &= & [q\omega_1+ \frac{1}{2} q(q-1)\omega^2_1]t^2_1 +q \omega_2t_2\\
P_{\omega,k,3}^q &=& [q\omega_1+ q(q-1) \omega^2_1  +\frac{1}{3!}q(q-1)(q-2)\omega^3_1]t^3_1 +[ q (\omega_1 + \omega_2) +  q (q-1)\omega_1 \omega_2] t_1t_2   + q \omega_3t_3
\end{array}$
\\and in the Stirling operator notation,  these translate into:

$\begin{array}{rcl}
P^q_{\omega, k,0} &=&1\\
P^q_{\omega,k,1} &= & B_0\omega_1 t_1\\
P_{\omega,k,2}^q &= & [B_0\omega_1+ \frac{1}{2} B_{-1}\omega^2_1]t^2_1 +B_0 \omega_2t_2\\
P_{\omega,k,3}^q &=& [B_0\omega_1+ B_{-1} \omega^2_1  +\frac{1}{3!}B_{-2}\omega^3_1]t^3_1 +[ B_0 (\omega_1 + \omega_2) +  B_{-1}\omega_1 \omega_2] t_1t_2   + B_0 \omega_3t_3
\end{array}$

\begin{remark} A more precise notation for the roots is  $P^{\ast q}$,  emphasizing that this root is to be taken with respect to the convolution product,  that is,  to retrieve the original function after have taken the $q$-th root, one must take the convolution product $\frac{1}{q} $  times.   We shall use the shorter form  $P^q$ with the meaning  $P^q = P^{\ast q}$.
\end{remark}


  In the next section,  we describe how GFPs are used to produce an isomorphism from the WIP-module to the group of multiplicative arithmetic functions \cite{LM1}.  


\section{Multiplicative Arithmetic Functions} \label{WIP-MF}

The ring  (UFD) of arithmetic functions consists of the functions  $\alpha:  \mathbb{Z} \to\mathbb{Q}$.  The \emph{Dirichlet product} of  two arithmetic functions $\alpha$ and $\beta$ is given by 
$$ \alpha \ast \beta (n) = \sum_{d} \alpha(d) \beta(\frac{n}{d}), $$
where $ d | n$ \cite{Mc}. 
 
 The \emph{multiplicative arithmetic functions (MF)} are those functions $\alpha$ such that $$ \alpha(mn) = \alpha(m)\alpha(n)$$  whenever  $(m,n)= 1$.    This is equivalent to saying that a multiplicative function is completely determined by its values at primes.  We shall say that such functions are determined locally,  so that we are interested in the products
  $$ \alpha \ast \beta (p^n) =  \sum^{n}_{i=0}\alpha(p^i) \beta(p^{n-i}).$$ 
If we consider the group generated by GFPs under convolution product as  multiplication,  then we also get an abelian group. And if we consider all of the evaluations of the variables $t_j$ over the integers,  we produce a group that is locally isomorphic to the group  MF \footnote{There are analogous results for the additive group  GLPs and the group of additive arithmetic functions.}  \cite{MW1,  LM1}.     It was shown that this  induces a mapping from the divisible closure of the group generated by GFPs to the divisible closure of  MF, and this mapping is a local isomorphism  \cite{MT1}.  

Thus the matrix representations of $F^{q}_{k,n}$ carry over to matrix representations of the divisible closure of MF.


\section*{References} 

\begin{thebibliography}{99}

\bibitem{ACMS} T. Amdeberhan, X. Chen, V. H. Moll and B. Sagan, Generalized Fibonacci polynomials and Fibonomial coefficients, \emph{Ann. Combin.}, 18 (2014), 541--562.

\bibitem{CG} T. B. Carroll and A. A. Gioia, On a subgroup of the group of multiplicative arithmetic
functions, \emph{J. Austral. Math. Soc.} (Ser. A) 20 (1975), 348--358.












\bibitem{LM1} H. Li and T. MacHenry, The convolution ring of arithmetic functions and symmetric
polynomials, \emph{Rocky Mountain J. Math.}, Vol. 43, No. 4, 1227--1259, (2013).
\bibitem{LM2} H. Li and T. MacHenry, Permanents and determinants, weighted isobaric polynomials and integral sequences, \emph{J. Integer Sequences}, Vol. 16, No. 3, 1--20, (2013).

\bibitem{M4} T. MacHenry, Generalized Fibonacci and Lucas polynomials and multiplicative arithmetic functions, \emph{Fibonacci Quarterly} 38 (2000), 17--24.
\bibitem{MT1} T. MacHenry and G. Tudose, Reflections on symmetric polynomials and arithmetic functions, \emph{Rocky Mountain J. Math.} 35 (2005), 901--928.
\bibitem{MT2} T. MacHenry and G. Tudose, Differential operators and weighted isobaric polynomials,
\emph{Rocky Mountain J. Math.} 36 (2006), 1957--1976.
\bibitem{MW1} T. MacHenry and K. Wong, Degree-$k$ linear recursions mod($p$) and algebraic number
fields, \emph{Rocky Mountain J. Math.} 41 (2011), 1303--1328.
\bibitem{MW2} T. MacHenry and K. Wong, A correspondence between isobaric rings and multiplicative arithmetic functions, \emph{Rocky Mountain J. Math.} 42 (2012), 1247--1290.
\bibitem{Mc} P. J. McCarthy, \emph{Introduction to Arithmetical Functions}, Springer-Verlag, New York, (1986). 
\bibitem{P} G. P\'olya, Kombinatorische  Anazahlbestimmungen f\"{u}r  Gr\"{u}ppen, Graphen und chemische Verbindungen,  \emph{Acta Math.} 68 (1937), 145--254.

\end{thebibliography}
\end{document}